\documentclass[10pt]{amsart}
\usepackage{amsfonts}
\usepackage{amssymb}
\usepackage{amsmath, amsthm, latexsym, tikz, array, MnSymbol, enumerate}
\usetikzlibrary{matrix,arrows,decorations.pathmorphing}
\usepackage[all]{xy}
\usepackage{hyperref}
\usepackage[margin=1.5in]{geometry}
\usepackage{color}

\def \C{\mathbb{C}}
\def \Z{\mathbb{Z}}
\def \R{\mathbb{R}}
\def \N{\mathbb{N}}
\def \P{\mathbb{P}}
\def \Q{\mathbb{Q}}
\def \GL{\textup{GL}}
\def \SL{\textup{SL}}
\def \FL{\mathcal{F}\ell}
\def \GZ{\textup{GZ}}
\def \PD{\textup{PD}}
\def \k{{\bf k}}

\theoremstyle{plain}
\newtheorem{Th}{Theorem}[section]
\newtheorem{Lem}[Th]{Lemma}
\newtheorem{Prop}[Th]{Proposition}
\newtheorem{Cor}[Th]{Corollary}

\newtheorem{THM}{Theorem}

\newtheorem{PROP}[THM]{Proposition}

\theoremstyle{definition}
\newtheorem{Ex}[Th]{Example}
\newtheorem{Def}[Th]{Definition}
\newtheorem{Rem}[Th]{Remark}

\def \Pic{\operatorname{Pic}}
\def \Vol{\operatorname{Vol}}

\begin{document}

\title{Cohomology ring of the flag variety vs Chow cohomology ring of the Gelfand-Zetlin toric variety}
\author{Kiumars Kaveh}
\address{Department of Mathematics, University of Pittsburgh,
Pittsburgh, PA, USA}
\email{kaveh@pitt.edu}
\author{Elise Villella}
\address{Department of Mathematics, University of Pittsburgh, Pittsburgh, PA, USA.}
\email{emv23@pitt.edu}

\thanks{The first author is partially supported by a National Science Foundation Grant (Grant ID: DMS-1601303).}
\date{\today}

\begin{abstract}
We compare the cohomology ring of the flag variety 
$\FL_n$ and the Chow cohomology ring of the Gelfand-Zetlin toric variety $X_\GZ$. We show that $H^*(\FL_n, \Q)$ is the Poincar\'e duality quotient of the subalgebra of $A^*(X_\GZ, \Q)$ generated by degree $1$ elements. We compute these algebras for $n=3$ and see that, in general, this subalgebra does not have Poincar\'e duality. 
%In particular, computation of intersection numbers in Schubert calculus can be converted to computation of intersection numbers of Minkowski weights on the Gelfad-Zetlin fan.
%We compute $A^*(X_\GZ, \Q)$ when $n=3$. 
\end{abstract}

\maketitle
\tableofcontents

%\begin{center}
%{\bf This is a preliminary version. Comments are welcome!} \vspace{1mm}
%\end{center}

\section*{Introduction}
Throughout the paper, the base field is assumed to be $\C$. The complete flag variety $\FL_n$ is the variety whose points parameterize complete flags of subspaces in $\C^n$, namely:
$$F = (\{0\} \subsetneqq F_1 \subsetneqq \cdots \subsetneqq F_n = \C^n).$$
The variety $\FL_n$ can be identified with the homogeneous space $\GL(n, \C) / B$ where $B$ is the subgroup of upper triangular matrices. The geometry of flag variety plays an important role in representation theory of $\GL(n, \C)$ and combinatorics related to the permutation group. More generally there is a notion of flag variety for any reductive algebraic group $G$. 

We recall that $\dim(\FL_n) = N = n(n-1)/2$. The classes of Schubert varieties form an important $\Z$-basis for $H^*(\FL_n, \Z)$. Since $\FL_n$ has a paving by affine cells (Schubert cells), it has no odd cohomology. Moreover, $H^*(\FL_n, \Z)$ is generated by degree $2$ elements. Also its Chow ring $A^*(\FL_n)$ is isomorphic to $H^*(\FL_n, \Z)$ where the isomorphism doubles the degree. The famous Borel description states that $H^*(\FL_n, \Z)$ is isomorphic to the polynomial algebra in $n$ variables quotient by the ideal generated by non-constant symmetric polynomials. 

We identify the weight lattice $\Lambda = \Lambda_{\GL(n, \C)}$ with the additive group $\Z^n$ and the semigroup of dominant weights $\Lambda^+ = \Lambda^+_{\GL(n, \C)}$ (respectively the positive Weyl chamber $\Lambda^+_\R$) with the collection of all increasing sequences $\lambda = (\lambda_1 \leq \cdots \leq \lambda_n)$ of integers (respectively real numbers). If $\lambda_1 < \cdots < \lambda_n$ we call $\lambda$ a regular dominant weight. We also denote the weight lattice $\Lambda(\SL(n, \C))$ of $\SL(n, \C)$ by $\Lambda'$. It can be identified with the quotient $\Lambda / \Z(1, \ldots, 1)$.

In their fundamental work \cite{Gelfand:1950ihs}, Gelfand and Zetlin\footnote{Warning to the reader: several different spellings of Zetlin's name appear in the English literature such as Tsetlin, Cetlin, Zeitlin or Tzetlin. Following Valentina Kiritchenko we use the spelling Zetlin, justified by the fact that while he was Russian his last name seems to have German origins.}  
construct a certain vector space basis $B_\lambda$ for an irreducible representation $V_\lambda$ of $\GL(n, \C)$ with highest weight $\lambda$, and they explicitly describe the action of $\mathfrak{gl}(n, \C) = \textup{Lie}(\GL(n, \C))$ on basis elements in $B_\lambda$. The Gelfand-Zetlin basis $B_\lambda$ has the remarkable property that its elements are indexed by the lattice points in a convex polytope $\Delta_\lambda \subset \R^N$, where $N=n(n-1)/2$, called the {\it Gelfand-Zetlin polytope} (or GZ polytope) associated to $\lambda$. The defining inequalities of $\Delta_\lambda$ can be explicitly written down. If $\lambda = (\lambda_1 \leq \cdots \leq \lambda_n)$ the polytope $\Delta_\lambda$ is the collection of $(x_{ij} \mid 1 \leq i \leq n-1, 1 \leq j \leq n-i) \in \R^N$ satisfying the following array of inequalities:
\begin{equation}\label{gzarray}
\begin{array}{ccccccccc}
\lambda_1 && \lambda_2 && \lambda_3 && \ldots && \lambda_n\\
&x_{11} && x_{12} && \ldots && x_{1(n-1)} &\\
&& x_{21} && x_{22} && \ldots &&\\
&&& \ddots && \udots &&&\\
&&&& x_{(n-1)1} &&&&
\end{array}
\end{equation}
where each small triangle $\begin{array}{ccc} a&&b\\ &c& \end{array}$ corresponds to the inequalities $a \leq c \leq b$. 
For example if $\lambda = (-1,0,1)$, the Gelfand-Zetlin polytope $\Delta_\lambda$ is given by the inequalities (see Figure \ref{gz1}): $$-1 \leq x \leq 0, \quad 0 \leq y \leq 1, \quad x \leq z \leq y.$$
\begin{figure}
\begin{center}
\includegraphics[scale=.5]{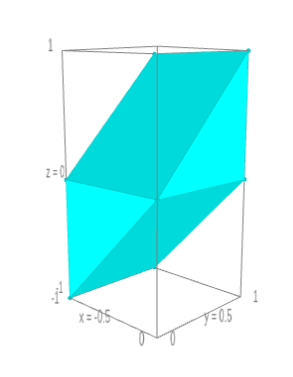}
\caption{Gelfand-Zetlin for $\lambda = (-1,0,1)$}\label{gz1}
\end{center}
\end{figure}

Since there is a one-to-one correspondence between the elements of the Gelfand-Zetlin basis $B_\lambda$ and the lattice points in $\Delta_\lambda$ one immediately sees that:
$$\dim(V_\lambda) = \#(\Delta_\lambda \cap \Z^N).$$
It is well-known that a weight $\lambda$ gives rise to a $\GL(n, \C)$-linearized line bundle $\mathcal{L}_\lambda$ on the flag variety $\FL_n$. When $\lambda$ is regular dominant the line bundle $\mathcal{L}_\lambda$ is very ample. By the Borel-Weil theorem $H^0(\FL_n, \mathcal{L}_\lambda) \cong V_\lambda^*$ as a $\GL(n, \C)$-module. Thus, in particular we have:
$$\dim(H^0(\FL_n, \mathcal{L}_\lambda)) = \#(\Delta_\lambda \cap \Z^N).$$

A general philosophy, suggested in the work of several authors and in particular A. Okounkov \cite{Okounkov-Newton-polytopes}, is that GZ polytopes play a role for the flag variety similar to that of Newton polytopes for toric varieties. In this direction in \cite{kaveh2011note} the first author obtains a description of $H^*(\FL_n, \Q)$ in terms of volumes of GZ polytopes. This description is very similar to the Khovanskii-Pukhlikov description of cohomology ring of a smooth projective toric variety in terms of volumes of Newton polytopes. The description in \cite{kaveh2011note} turns out to be equivalent to the Borel description via a theorem of Kostant (see \cite[Remark 5.4]{kaveh2011note}). Making the connection between geometry of $\FL_n$ and GZ polytopes stronger, in \cite{KST} the authors make a correspondence between Schubert varieties and certain unions of faces of GZ polytopes. They use this correspondence to give applications in Schubert calculus. 

It can be shown that for regular dominant weights $\lambda$, all the polytopes $\Delta_\lambda$ have the same normal fan (Proposition \ref{prop-GZ-fan}). We call this common normal fan the {\it Gelfand-Zetlin fan} and denote it by $\Sigma_\GZ$. It is well-known that, for each regular dominant $\lambda$ the pair $(\FL_n, \mathcal{L}_\lambda)$ can be degenerated, in a flat family with reduced irreducible fibers, to $(X_\GZ, \mathcal{L}_{\Delta_\lambda})$. Here $\mathcal{L}_{\Delta_\lambda}$ is the equivariant line bundle on the toric variety $X_\GZ$ corresponding to the lattice polytope $\Delta_\lambda$ (see \cite{Kogan-Miller}). Such degenerations have been used to study mirror symmetry for the flag variety and partial flag varieties (see \cite{Batyrev}). This motivates the problem of comparing the geometry and topology of $\FL_n$ with that of $X_\GZ$. 

The variety $X_\GZ$ is not smooth and hence its Chow group does not have a ring structure. There is a dual version of the Chow ring, due to Fulton and MacPherson \cite{fulton1981categorical}, that works for singular varieties as well. It is called the {\it operational Chow ring} or simply {\it Chow cohomology ring}. For a variety $X$ we denote its Chow cohomology ring by $A^*(X)$. 

Let ${\bf k}$ be a field. Given a graded algebra $A=\bigoplus_{i=0}^n A^i$ with $A^0 \cong A^n \cong {\bf k}$, one can form the largest quotient $A/I$ of $A$ such that $A/I$ has Poincar\'e duality (Lemma \ref{pd-lemma}). We call this the {\it Poincar\'e duality quotient of $A$} and denote it by $\PD(A)$. The main result of the paper is the following (Theorem \ref{th-main}):

\begin{THM}   \label{th-intro-main}
The cohomology ring $H^*(\FL_n, \Q)$ is isomorphic to the Poincar\'e duality quotient of the subalgebra of $A^*(X_\GZ, \Q)$ generated by degree $1$ elements.  
\end{THM}

One key combinatorial ingredient in the proof is the following statement suggested to us by  Valentina Kiritchenko (Proposition \ref{normalfan}):
\begin{PROP}
Let $P$ be a polytope whose normal fan is $\Sigma_{GZ}$, then $P = c + \Delta_\lambda$ for some $\lambda \in \Lambda^+$ and $c \in \R^N$.
\end{PROP}
 
Another ingredient in the proof of Theorem \ref{th-intro-main} is an algebra lemma which states that a Poincar\'e duality algebra $A = \bigoplus_{i=0}^n A^i$ that is finite dimensional as a vector space and is generated (over $A^0$) by $A^1$, is uniquely determined by its top product polynomial $p: A^1 \to A^n \cong A^0$, $p(x) = x^n$ (Theorem \ref{algebraKaveh}).

In \cite{fulton1997intersection} it is shown that the Chow cohomology ring of a toric variety $X_\Sigma$ is naturally isomorphic to the ring of {\it Minkowski weights} on its fan $\Sigma$. A degree $k$ Minkowski weight on a fan $\Sigma$ is an assignment of integers to $k$-dimensional cones in $\Sigma$ which satisfies certain balancing condition. One defines a product of Minkowski weights that makes the collection of all Minkowski weights into a ring (see Section \ref{sec-MW}, see also \cite{fulton1997intersection, Kazarnovskii}). 
There is also an alternative description of the Chow cohomology ring of a toric variety in terms of piecewise linear functions on its fan (see \cite{Payne}). 

In Section \ref{sec-n=3}, we use the Minkowski weights description of the Chow cohomology ring, to compute $A^*(X_\GZ, \Q)$ for $n=3$ and see directly that it coincides with its subalgebra generated by degree $1$ elements. We also see $A^*(X_\GZ, \Q)$ does not have Poincar\'e duality. %use the Minkowski weights description of the Chow cohomology as it is easier to work with for computational purposes.

The second author has written a Sage code that verifies that for $n = 4, 5$ the Chow cohomology ring of $X_\GZ$ is not generated in degree $1$, and moreover the subalgebra generated in degree $1$ does not have Poincar\'e duality. See {\sf https://github.com/evillella/minkowski}. Also see the appendix in the second author's Ph.D. thesis \cite{EliseThesis}.
%To make the manuscript accessible to a wider range of audience we have tried to include more background material.

In geometric terms, the isomorphism between the Picard groups of $\FL_n$ and $X_\GZ$ can be constructed by means of a \emph{toric degeneration}. A toric degeneration of $\FL_n$ to $X_\GZ$ is a flat family $\pi: \mathcal{X} \to \C$ with reduced fibers and an action of $\C^*$ lifting the $\C^*$ action on the base $\C$ such that the general fiber $X_t := \pi^{-1}(t)$, $t \neq 0$, is $\FL_n$ and its unique special fiber $X_0 := \pi^{-1}(0)$ is $X_\GZ$. Then any divisor class $[D]$ in $\Pic(\FL_n)$ can be extended to the whole family $\mathcal{X}$ and then specialized to the special fiber $X_\GZ$ to get a divisor class $[D_0]$ on $X_\GZ$. For a general toric degeneration, $[D_0]$ may not be a Cartier divisor class. But one shows that this is the case for example for the family constructed in \cite{Kogan-Miller} (see \cite[Proposition 11]{Kogan-Miller}). In fact, under this specialization map the class of a line bundle $\mathcal{L}_\lambda$ on $\FL_n$ goes to the class of the line bundle on $X_\GZ$ determined by the polytope $\Delta_\lambda$. We do not know if this construction extends to give a homomorphism between the Chow cohomology rings.\\ %Give a reference as well as check if it is $\Pic(X_\GZ)$ or $A^1(X_\GZ)$? 

\noindent{\bf Acknowledgements}
We thank Valentina Kiritchenko and June Huh for very helpful email correspondence and conversations. We are also in debt to Kyeong-Dong Park for proof reading a first version of the paper and giving useful comments.
%We are also thankful to June Huh for brief but beneficial converstaion. 
%Throughout the base field is complex numbers. The paper is written with the intention of being accessible to a wider range of audience and hence we have tried to cover most of the background material. 
%Introduce flag variety, Borel-Weil theorem, cohomology of flag variety, Schubert calculus and enumerative geometry, GZ polytope, degree of a line bundle and volume of a GZ polytope, description of cohomology ring in terms of volume of GZ polytopes \cite{Kaveh}, Schubert calculus and GZ polytopes \cite{KST}, GZ degeneration \cite{Kogan-Miller}, Chow cohomology ring, natural question: how the intersection theory of flag variety compares with that of GZ toric variety? We answer this question in this paper. Poincar\'e duality quotient, subalgebra generated by degree $1$ elements, main theorem.

\section{Some facts about Gelfand-Zetlin polytopes}
%We recall the definition of a Gelfand-Zetlin polytope (GZ polytope) originally defined in \cite{Gelfand:1950ihs} and studied by many people since, see for example \cite{kiritchenko}. For $\lambda = (\lambda_1 < \lambda_2 < \ldots < \lambda_n)$ the GZ polytope $\Delta_\lambda$ is defined by the following array of inequalities:
%\begin{equation}\label{gzarray}
%\begin{array}{ccccccccc}
%\lambda_1 && \lambda_2 && \lambda_3 && \ldots && \lambda_n\\
%&x_{11} && x_{12} && \ldots && x_{1(n-1)} &\\
%&& x_{21} && x_{22} && \ldots &&\\
%&&& \ddots && \udots &&&\\
%&&&& x_{n1} &&&&
%\end{array}
%\end{equation}
%where each small triangle $\begin{array}{ccc} a&&b\\ &c& \end{array}$ corresponds to the inequalities $a \leq c \leq b$. Let $N = n(n-1)/2$ be the dimension of $\Delta_\lambda$. 
In this section we prove some basic facts about GZ polytopes. We start with the normal fan to a GZ polytope $\Delta_\lambda$. 
%From the definition of normal fan, a $k$-dimensional cone $\sigma \in \Sigma_\lambda$ corresponds to a face $F \subset \Delta_\lambda$ of dimension $N-k$ and inclusion of faces and cones correspond, though are reversed.
Recall that the normal fan $\Sigma_\Delta$ of a polytope $\Delta$ is constructed as follows: for each face $F$ let $C_F$ be the face cone of $F$ and let $\sigma_F$ be the dual cone to $C_F$. Then $\Sigma_\Delta = \{\sigma_F \mid F \textup{ face of } \Delta\}$ (see \cite[Section 2.3]{Cox-Little-Schenck}).

\begin{Prop}   \label{prop-GZ-fan}
For a regular dominant weight $\lambda$, the normal fan $\Sigma_{\lambda}$ of $\Delta_\lambda$ is independent of $\lambda$.
\end{Prop}
\begin{proof} %We recall the coordinates given in the GZ array \eqref{gzarray}. 
The facets of $\Delta_\lambda$ correspond to single equalities in the array \eqref{gzarray}, and lower dimensional faces of $\Delta_\lambda$ correspond to multiple equalities in the array. There are two types of equality that can occur, (i) those of the form $x_{1i} = \lambda_j$ and (ii) those of the form $x_{ij} = x_{(i-1)k}$. The second type of equality is clearly independent of $\lambda$ and the first type depends on $\lambda$, so that the faces corresponding to various $\lambda$ values differ only by translation. It follows that for a face $F$, which is defined by a combination of these two types of equalities, the corresponding face cone $C_F$ and hence its dual cone $\sigma_F$ is independent of $\lambda$. This proves the claim.  
%Fix a face of $F$ of $\Delta_\lambda$, that is, fix a collection of equalities in the array. Then the cone at that face as $\lambda$ varies is simply translated based on how many equalities of the first type appear in the array. When we examine the corresponding cone in the fan $\Sigma_\lambda$, we translate the cone at the face $F$ to the origin then take the dual cone. Thus the fan $\Sigma_\lambda$ does not depend on $\lambda$. 
\end{proof}

\begin{Def}[Gelfand-Zetlin fan]   \label{def-GZ-fan}
We refer to the common normal fan of the $\Delta_\lambda$, where $\lambda$ is regular dominant, as the {\it Gelfand-Zetlin fan} and denote it by $\Sigma_\GZ$.
\end{Def}

%We say that a polytope $P$ is {\it normal} to a fan $\Sigma$ if $\Sigma$ is a refinement of the normal fan of $P$.
\begin{Prop}\label{normalfan}
Let $P$ be a polytope whose normal fan is $\Sigma_{GZ}$, then $P = c + \Delta_\lambda$ for some $\lambda \in \Lambda^+$ and $c \in \R^N$. Moreover, if $P$ is a lattice polytope then $\lambda$ is a dominant weight and $c \in \Z^N$.
\end{Prop}

\begin{proof}
Since normal fan of $P$ is $\Sigma_{GZ}$, the hyperplanes defining $P$ are parallel to the ones defining $\Delta_\lambda$, for any dominant regular $\lambda$ (as we have already showed the fan is independent of $\lambda$). 
Let us use $y_{ij}$ (respectively $x_{ij}$) for coordinates of a point in $P$ (respectively a GZ polytope $\Delta_\lambda$). 
Recall that there are two types of inequalities defining $\Delta_\lambda$ namely, $\lambda_{i} \leq x_{1i} \leq \lambda_{i+1}$ and $x_{(i-1)j} \leq x_{ij} \leq x_{(i-1)(j+1)}$. 
Since the facets of $P$ are parallel to those of a GZ polytope we conclude that the inequalities defining $P$ come in two types as well: %$y_{1i} \leq a_j$ and $y_{ij} \leq y_{(i-1)k}$ for appropriate $i, j, k$. Then in terms of inequalities, $P$ is given by {\bf (Check $i=1,\ldots, n-1$?)}
\begin{align}\label{P1}
    a_j &\leq y_{1j} \leq b_j \quad &1 \leq j \leq n-1 \\
    y_{(i-1)j} + a_{ij} & \leq y_{ij} \leq y_{(i-1)(j+1)} + b_{ij} \quad &2 \leq i \leq n-1, \;  1 \leq j \leq n-i.\nonumber
\end{align}
%\begin{align}\label{P1}
%    a_i &\leq y_{1i} \leq b_i \quad & 1 \leq i \leq n-1 \\
%    y_{(i-1)j} + a_{ij} & \leq y_{ij} \leq y_{(i-1)(j+1)} + b_{ij} \quad &2 \leq i \leq n, \;  1 \leq j \leq n-i+1.\nonumber
%\end{align}
%If by translating $(y_{11},y_{12},\ldots,y_{n1})$ we can fit the above inequalities \eqref{P1} into an array such as \eqref{gzarray}, then we will have shown that any $P$ with $\Sigma_{GZ}$ as normal fan is in fact a translation of a GZ polytope. In other words, 
We wish to find $\lambda = (\lambda_1 \leq \cdots \leq \lambda_n)$ and $c = (c_{ij}) \in \R^N$ such that if $x_{ij} = y_{ij} + c_{ij}$ then the inequalities \eqref{P1} for the $y_{ij}$ are equivalent to the GZ inequalities \eqref{gzarray} for the $x_{ij}$.

The first type of inequalities $a_j \leq y_{1j} \leq b_j$ tell us what $\lambda$ to choose. Set $\lambda_1 = a_1$ and $\lambda_2 = b_1$. 
By induction suppose for $1 \leq j < n-1$ we have picked $\lambda_1,\ldots, \lambda_{j+1}$ and $c_{11} = 0, c_{12}, \ldots, c_{1j}$ such that:
$$\lambda_1 \leq x_{11} = y_{11} \leq \lambda_2 \leq x_{12} \leq \ldots \leq x_{1j} \leq \lambda_{j+1},$$ where $x_{1k} = y_{1k} + c_{1k}$, for all $k$. Now if we put $c_{1(j+1)} = \lambda_{j+1}-a_{j+1}$ and $\lambda_{j+2} = b_{j+1}+\lambda_{j+1}-a_{j+1}$ we have $\lambda_{j+1} \leq x_{1(j+1)} \leq \lambda_{j+2}$ as required.
%By induction, if we have $$\lambda_1 \leq x_{11} = y_{11} \leq \lambda_2 \leq x_{12} \leq \ldots \leq x_{1i} \leq \lambda_{i+1}$$ then we can translate $y_{1(i+1)}$ as follows. We are given $a_{i+1} \leq y_{1(i+1)} \leq b_{i+1}$ and want to shift $y_{1(i+1)}$ to some $x_{1(i+1)}$ satisfying $\lambda_{i+1} \leq x_{1(i+1)}$. This can be done by setting $$x_{1(i+1)} = y_{1(i+1)}+\lambda_{i+1}-a_{i+1},$$ and this also implies that $$\lambda_{i+2} = b_{i+1}+\lambda_{i+1}-a_{i+1},$$ so that $\lambda_{i+1} \leq x_{1(i+1)} \leq \lambda_{i+2}$. Thus we can translate the variables $y_{1i}$ by $c_{1i} = \lambda_{i+1}-a_{i+1}$ to fit into the first row of an array of the form \eqref{gzarray}.

For the remaining rows, we first need to examine the small diamonds $\begin{array}{ccc} &a& \\b&&c\\ &d& \end{array}$ appearing in the GZ array \eqref{gzarray}. Since $b \leq d \leq c$, the equalities $b=a$ and $c=a$ imply $d=a$. This gives us linear relations among the ray generators in the fan $\Sigma_\GZ$ which in turn translate to relations among the $a_{ij}$, $b_{ij}$ for the polytope $P$. Let $1 < i < n-1$ and $1 \leq j \leq n-i$ and by induction suppose we have picked $c_{11}, \ldots, c_{i(j-1)}$ so that $x_{11}, \ldots, x_{i(j-1)}$ satisfy the GZ triangular array of inequalities. We would like to find $c_{ij}$ so that $x_{ij} = y_{ij} + c_{ij}$ satisfies the diamond $$\begin{array}{ccc} &x_{(i-2)(j+1)}& \\x_{(i-1)j}&&x_{(i-1)(j+1)}\\ &x_{ij}& \end{array}$$
%Suppose we have translated variables $y_{11}, \ldots, y_{i(j-1)}$ to fit into a triangular array of inequalities, then the diamond relation will allow us to shift $y_{ij}$ so that the associated inequalities will fit into the array as well. Our goal is to define $x_{ij}$ so that it fits into the diamond $$\begin{array}{ccc} &x_{(i-2)(j+1)}& \\x_{(i-1)j}&&x_{(i-1)(j+1)}\\ &x_{ij}& \end{array}$$ except in the case $i=2$ where we have $\lambda_{j+1}$ instead of $x_{(i-2)(j+1)}$. 
The second type of inequalities in \eqref{P1} can be written as:
\begin{equation}\label{relation}
x_{(i-1)j}+a'_{ij} \leq y_{ij} \leq x_{(i-1)(j+1)} + b'_{ij}
\end{equation}
where $a'_{ij} = a_{ij} + c_{(i-1)j}$ and $b'_{ij} = b_{ij} + c_{(i-1)(j+1)}$.
Now when we consider the face of $P$ where $x_{(i-1)j} = x_{(i-2)(j+1)}$ and $x_{(i-1)(j+1)}=x_{(i-2)(j+1)}$, by what we said above, the inequality \eqref{relation} becomes two equalities. We thus have: 
%\begin{equation}\label{relation}
%x_{(i-1)j}+a'_{ij} \leq y_{ij} \leq x_{(i-1)(j+1)} + b'_{ij}
%\end{equation}
%which by the diamond relation becomes two equalities rather than inequalities. This then gives:
$$x_{(i-2)(j+1)} + a'_{ij} = x_{(i-2)(j+1)} + b'_{ij}$$ which implies $a'_{ij} = b'_{ij}$. Now, if we define $x_{ij} = y_{ij} - a'_{ij}$, i.e. $c_{ij}=-a'_{ij}$, the relation \eqref{relation} becomes $$x_{(i-1)j} \leq x_{ij} \leq x_{(i-1)(j+1)}$$ as required. %Now that we have fit $x_{ij} = y_{ij}-a_{ij}$ into the triangular array, translate the $a_{(i+1)(k)}, b_{(i+1)k}$ so that any relations previously involving $y_{ij}$ now contain $x_{ij}$ instead. 
%Thus we have fit the $x_{ij}$ into the triangular array. 
%In this way we can shift all the variables for polytope $P$ such that the defining inequalities fit into a triangular array. 
Therefore $P = c + \Delta_\lambda$ where $\lambda = (\lambda_1, \ldots, \lambda_n)$ and $c = (c_{ij})$ as constructed above. Finally, if $P$ is a lattice polytope then the $a_i$, $b_i$, $a_{ij}$, $b_{ij}$ should be integers (note that none of the inequalities in \eqref{P1} is redundant and the corresponding equality defines a facet of $P$). This implies that $\lambda$ and $c$ are integer vectors as well.  
\end{proof}

\begin{Rem}   \label{rem-Kiritchenko}
Proposition \ref{normalfan} was suggested to us by Valentina Kiritchenko. The proof presented above is due to the second author.
\end{Rem}

\begin{Rem}
Observe that there are $n+ n(n-1)/2$ parameters present in $c + \Delta_\lambda$, but a GZ polytope is cut out by $n(n-1)$ facets, one for each ray in $\Sigma_\GZ(1)$. The dimension of the space of polytopes with normal fan $\Sigma_{GZ}$ is hence much smaller than the number of rays in the fan due to the fact that $\Delta_\lambda$ is not a simple polytope, or equivalently, the fan $\Sigma_{GZ}$ is not simplicial.
\end{Rem}

A third useful property of the GZ polytopes is that they behave well with respect to Minkowski addition. We recall that for polytopes $P$ and $Q$, the \emph{Minkowski sum} $P+Q$ is the polytope $$P+Q = \{x+y \ | x \in P, y \in Q\}.$$

\begin{Prop}\label{additive}
The assignment $\lambda \mapsto \Delta_\lambda$ is additive, that is, for any dominant weights $\lambda, \mu \in \Z^n$ we have: 
$$\Delta_{\lambda + \mu} = \Delta_\lambda + \Delta_\mu,$$ where the addition on the right is the Minkowski sum of polytopes.
\end{Prop}

\begin{proof}
The inclusion $\Delta_\lambda + \Delta_\mu \subset \Delta_{\lambda+\mu}$ is clear. We need to show the other direction. 
Let $x \in \Delta_{\lambda+\mu}$, our goal is to write $x = x' + x''$ with $x' \in \Delta_\lambda$ and $x'' \in \Delta_\mu$. We begin with the first row $x_{1*} = (x_{11}, \ldots, x_{1(n-1)})$ satisfying 
$$\lambda_1 + \mu_1 \leq x_{11} \leq \lambda_2 + \mu_2 \leq \cdots \leq x_{1(n-1)} \leq \lambda_n + \mu_n.$$
It is clear that, for each $i$, the sum of line segments $[\lambda_i, \lambda_{i+1}]$ and $[\mu_i, \mu_{i+1}]$ is $[\lambda_i + \mu_i, \lambda_{i+1}+\mu_{i+1}]$. Thus we can find $x'_{1*}, x''_{1*} \in \R^{n-1}$ such that $x_{1*} = x'_{1*} + x''_{1*}$ and they satisfy the first row of interlacing inequalities for $\Delta_\lambda$ and $\Delta_\mu$ respectively. We can then repeat the same argument for the second row replacing $\lambda$, $\mu$ with $x'_{1*}$, $x''_{1*}$ to obtain $x'_{2*}, x''_{2*} \in \R^{n-2}$. Continuing this way we find $x' \in \Delta_\lambda$, $x'' \in \Delta_\mu$ with $x= x'+x''$ as required.  
%We prove this by induction on $n$. Let $x \in \Delta_{\lambda+\mu}$, our goal is to write $x = x' + x''$ with $x' \in \Delta_\lambda$ and $x'' \in \Delta_\mu$. We begin with the first row of inequalities, $\lambda_1 + \mu_1 \leq x_{11} \leq \lambda_2 + \mu_2 \leq \cdots$. This can be reduced to a number of inequalities of the form $$0 \leq y \leq a+b$$ with appropriate definitions of $a, b, y$. We can then write $y = y' + y''$ where $y' = y \frac{a}{a+b}$, $y'' = y \frac{b}{a+b}$.   Note that $0 \leq y' \leq a$, $0 \leq y'' \leq b$. %For this, we let $$y' = y \frac{a}{a+b}, \quad y'' = y \frac{b}{a+b},$$ then $y = y' + y''$. By assumption $y, a, b$ are all positive, so clearly $y', y'' \geq 0$. We just need to show $y' \leq a$ and $y'' \leq b$. We have $y \leq a+b$ so $$y' = y \frac{a}{a+b} \leq (a+b) \frac{a}{a+b} = a$$ and similarly $y'' \leq b$ as desired.
%Now, in the top row of GZ array, we convert each $\lambda_i + \mu_i \leq x_{1i} \leq \lambda_{i+1} + \mu_{i+1}$ into $0 \leq y \leq a+b$ by taking $$y = x_{1i} - \lambda_i - \mu_i, \quad a = \lambda_{i+1}-\lambda_i, \quad b=\mu_{i+1} - \mu_i.$$ These quantities are all positive and satisfy the desired inequality following directly from our assumption.  Thus we can write the $x_{1i}$ into $x'_{1i} + x''_{1i}$, we then continue inductively treating the $x_{1i}' + x_{1i}''$ terms as we did the $\lambda_i + \mu_i$ to give a splitting for the next row. Thus we have shown that the assignment $\lambda \mapsto \Delta_\lambda$ is in fact additive.
\end{proof}

\begin{Rem}
Proposition \ref{additive} shows that the collection of Gelfand-Zetlin polytopes is an example of a \emph{linear family of polytopes} (as defined in \cite{kaveh2018notion}). In this regard, Proposition  \ref{prop-GZ-fan} is related to \cite[Proposition 1.3]{kaveh2018notion}.
\end{Rem}

\begin{Prop}  \label{prop-injective}
Suppose for two dominant weights $\lambda$, $\lambda' \in \Lambda$ and $c \in \Z^N$ we have $c + \Delta_{\lambda} = \Delta_{\lambda'}$. Then $\lambda - \lambda'$ is a multiple of $(1,\ldots,1)$, that is, $\lambda$, $\lambda'$ represent the same weight in $\Lambda'$.
\end{Prop}
\begin{proof}
Let the $(x_{ij})$, $(x'_{ij})$ denote the coordinates of points in $\Delta_{\lambda}$, $\Delta_{\lambda'}$ respectively. Also let $c = (c_{ij})$. The assumption that $c + \Delta_{\lambda} = \Delta_{\lambda'}$ implies that for all $1 \leq i \leq n-1$, $\lambda_i \leq x_{1i} \leq \lambda_{i+1}$ if and only if $\lambda'_i \leq x_{1i} + c_{1i} \leq \lambda'_{i+1}$. It follows that $\lambda'_i = \lambda_i + c_{1i}$ and $\lambda'_{i+1} = \lambda_{i+1} + c_{1i}$ which in turn implies that $c_{1i} = c_{1(i+1)}$. This finishes the proof. 
\end{proof}

Recall that a virtual polytope is a formal difference of two polytopes. The set of virtual polytopes in $\R^N$ forms an infinite dimensional $\R$-vector space. For a fan $\Sigma$ in $\R^N$ let $\mathcal{P}(\Sigma)$ denote the subgroup of virtual lattice polytopes in $\R^N$ generated by polytopes whose normal fan is $\Sigma$. The group $\mathcal{P}(\Sigma)$ contains a copy of the additive group $\Z^N$ as the virtual lattice polytopes whose support function is linear on the whole $\R^N$.
\begin{Cor}    \label{cor-additive}
(1) The map $\lambda \mapsto \Delta_\lambda$ gives a homomorphism $\phi: \Lambda = \Lambda(\GL(n, \C)) \to \mathcal{P}(\Sigma_\GZ)$. (2) The homomorphism  $\phi$ induces an isomorphism $\bar{\phi}: \Lambda' = \Lambda(\SL(n, \C)) = \Lambda / \Z(1, \ldots, 1) \to \mathcal{P}(\Sigma_\GZ) / \Z^N$. (3) The quotient group $\Lambda'$ is isomorphic to the Picard group of the toric variety $X_\GZ$ associated to the fan $\Sigma_\GZ$.
\end{Cor}
\begin{proof}
The assertion (1) is an immediate corollary of Proposition \ref{additive}. To prove (2) note that surjectivity of $\bar{\phi}$ follows from Proposition \ref{normalfan} and the injectivity of $\bar{\phi}$ is the content of Proposition \ref{prop-injective}. Finally, (3) follows from the well-known fact that for a fan $\Sigma$, the group $\Pic(X_\Sigma)$ is isomorphic to the group $\textup{PL}(\Sigma, \Z^N)$ of integer piecewise linear functions on $\Sigma$ modulo integer linear functions. This in turn can be identified with the quotient group $\mathcal{P}(\Sigma) / \Z^N$ (see \cite[Theorem 4.2.12]{Cox-Little-Schenck}).
\end{proof}

\section{Review of degrees of line bundles on toric and flag varieties}\label{degree}
We recall that, for a projective variety $X$ of dimension $d$ embedded into a projective space $\P^s$, the \emph{degree} of $X$ is defined to be: $$\text{deg}(X) = \# (X \cap H_1 \cap \ldots \cap H_d),$$ where the $H_i$ are generic hyperplanes in $\P^s$. 
%Alternatively, this is the intersection number $X \cdot H^{s-d}$. %As the Picard group $\text{Pic}(\P^s) \cong \Z$, we can choose a single hyperplane $H$ corresponding to the generator of $\text{Pic}(\P^s)$ and compute $\# (X \cap H \cap \ldots \cap H)$ in $\P^s$. 
Alternatively, let $[H]$ be the class of a hyperplane in $\text{Pic}(\P^s) \cong \Z$ and let $[H']$ be the pullback of $[H]$ to $X$ via the embedding $X \hookrightarrow \P^s$, then $\text{deg}(X) = [H']^d$, the self-intersection number of the divisor class $[H']$.

If the embedding $X \hookrightarrow \P^s$ is given by the sections of a very ample line bundle $\mathcal{L}$, that is, $X \hookrightarrow \P(H^0(X,\mathcal{L})^*)$, we will write $\deg(X, \mathcal{L})$ for $\deg(X)$. The asymptotic Riemann-Roch theorem, implies that $$\text{deg}(X, \mathcal{L}) = d! \lim_{m \to \infty} \frac{\dim H^0(X,\mathcal{L}^{\otimes m})}{m^d}.$$
If $\mathcal{L}$ is not very ample, we still define $\deg(X, \mathcal{L})$ as the self-intersection number of the divisor class of $\mathcal{L}$.

In the case $X = X_\Sigma$ is the toric variety of a fan $\Sigma$, we recall that all divisors are linearly equivalent to $T$-invariant divisors which in turn are generated by codimension $1$ orbit closures $D_\rho = \overline{O}_\rho$, $\rho \in \Sigma(1)$. Thus an arbitrary $T$-invariant divisor on $X_\Sigma$ can be written in the form $D = \sum_{\rho} a_\rho D_\rho.$ The associated line bundle will be $\mathcal{L} = \mathcal{O}(D)$, and the dimension of $H^0(X,\mathcal{L})$ is equal to the number of lattice points in the polytope $P_D = \{m \ | \ \langle m, v_\rho \rangle \geq -a_\rho, \, \forall \rho \in \Sigma(1) \}$ where $v_\rho$ is the primitive vector along the ray $\rho$.
One can also start with a lattice polytope $P$ normal to the fan of $X_\Sigma$. The \emph{support numbers} $\{ a_\rho \}_{\rho \in \Sigma(1)}$ of the polytope enable us to define a $T$-invariant divisor $D_P = \sum_{\rho \in \Sigma(1)} a_\rho D_\rho$ on $X_\Sigma$, and $P_{D_{P}} = P$. One shows that $D$ is ample that is, $kD$ defines an embedding into projective space for sufficiently large $k \in \N$. We have the following (which is a version of the well-known Bernstein-Kushnirenko-Khovanskii theorem):
\begin{Prop}\label{degreeToric}
Let $\mathcal{L}_P$ be the line bundle associated to the divisor $D_P$. Then: $$\deg(X_\Sigma, \mathcal{L}_P) = d! \Vol_d(P).$$
\end{Prop}
\begin{proof} By the asymptotic Riemann-Roch we have:
$$\text{deg}(X_\Sigma, \mathcal{L}_P) = d! \lim_{m \to \infty} \frac{\dim H^0(X_\Sigma,\mathcal{L}_P^{\otimes m})}{m^d}= d! \lim_{m \to \infty} \frac{\#(mP \cap \Z^d)}{m^d} = d! \Vol_d(P).$$
    %\begin{align*}
%    \text{deg}(X_\Sigma, \mathcal{L}_P) &= d! \lim_{m \to \infty} \frac{\dim H^0(X_\Sigma,\mathcal{O}(mP))}{m^d}\\
%    &= d! \lim_{m \to \infty} \frac{\#(mP \cap \Z^d)}{m^d}\\
%    &= d! \text{Vol}_d(P).
%\end{align*}
\end{proof}

{As we are interested in comparing $X_{\GZ}$ with the flag variety $\FL_n$, we also recall some facts about degrees of embeddings for $\FL_n$. Recall that to a weight $\lambda$ one associates a line bundle $\mathcal{L}_\lambda$ on $\FL_n$. This line bundle satisfies the property $$\mathcal{L}_\lambda^{\otimes m} = \mathcal{L}_{m\lambda}.$$ 
%and, from the Borel-Weil Theorem, $$H^0(G/B,\mathcal{L}_\lambda) \cong V_\lambda^*.$$ From the original construction by Gelfand and Zetlin in \cite{Gelfand:1950ihs}, a basis is constructed for this vector space with one basis element associated to each lattice point in $\Delta_\lambda$. Hence
%\begin{equation}
%    \dim V_\lambda^* = \# (\Delta_\lambda \cap \Z^N) = \dim H^0(\FL_n,\mathcal{L}_{\lambda}).
%\end{equation}
}
Similarly to the proof of Proposition \ref{degreeToric}, we can show the following (see for example \cite[Remark 2.4]{kaveh2011note}).
\begin{Prop}\label{flagDegree}
For any dominant weight $\lambda$ we have:
$$\deg(\FL_n,\mathcal{L}_\lambda) = N! \Vol_N(\Delta_\lambda),$$ where $N = n(n-1)/2 = \dim(\FL_n)$.
\end{Prop}
\begin{proof}
By the construction of the Gelfand-Zetlin polytope \cite{Gelfand:1950ihs}, for every dominant $\lambda$, we have $\# (\Delta_\lambda \cap \Z^N) = \dim(V_\lambda) = \dim(V_\lambda^*)$. On the other hand, by the Borel-Weil theorem, one knows that 
$H^0(\FL_n, \mathcal{L}_\lambda) \cong V_\lambda^*$. We note that for any $m>0$, $\mathcal{L}_\lambda^{\otimes m} = \mathcal{L}_{m\lambda}$ and $m\Delta_\lambda = \Delta_{m\lambda}$. Then the asymptotic Riemann-Roch theorem gives us:
$$\text{deg}(\FL_n, \mathcal{L}_\lambda) = N! \lim_{m \to \infty} \frac{\dim H^0(\FL_n,\mathcal{L}_\lambda^{\otimes m})}{m^N}= N! \lim_{m \to \infty} \frac{\#(m\Delta_\lambda \cap \Z^N)}{m^N} = N! \Vol_N(\Delta_\lambda),$$ as required.   
\end{proof}

Proposition \ref{flagDegree} and Proposition \ref{degreeToric} show that the map $\Pic(\FL_n) \to \Pic(X_\GZ)$, given by $\mathcal{L}_\lambda \mapsto \mathcal{L}_{\Delta_\lambda}$, preserves degree of line bundles. This observation is important in the proof of our main theorem (Theorem \ref{th-main}).

%for P_D_P = P see Thm 6.2.1 of Cox, divisor of polytope always cartier and ample, see also Cox thm 13.4.1 and 13.4.3

\section{Review of intersection theory on toric and flag varieties}
In this section we recall some basic facts about Chow rings and Chow cohomology rings of toric and flag varieties.
%recall the definitions of Chow groups and rings and, for the case of smooth toric varieties, Chow cohomology.

%\begin{Def}
%For a toric variety $X_\Sigma$, the Chow group $A_k(X_\Sigma)$ is generated by orbit closures $V(\sigma)$ for $\sigma \in \Sigma$ of codimension $k$.
%\end{Def}

For an algebraic variety $X$ and $1 \leq k \leq n = \dim(X)$, the $k$-th Chow group $A_k(X)$ is the group generated by algebraic $k$-cycles on $X$, that is, formal sums of irreducible $k$-dimensional subvarieties in $X$, modulo rational equivalence. Two $k$-cycles are equivalent if their difference is the divisor of a rational function on a $(k+1)$-dimensional subvariety, and the rational equivalence is the  equivalence relation generated by this. The total Chow group of $X$ is $A_*(X)=\bigoplus_{k=0}^n A_k(X)$. When $X$ is smooth we let $A^k(X) = A_{n-k}(X)$ and $A^*(X) = \bigoplus_{k=0}^n A^k(X)$. 
In this case, the transverse intersection of subvarieties gives a well-defined multiplication on $A^*(X)$ making it into a graded algebra called the Chow ring of $X$ (\cite[Proposition 8.3]{fulton2013intersection}). More generally, for a commutative ring $R$, one can define the Chow groups $A_k(X, R)$ and the Chow ring $A^*(X, R)$ whenever $X$ is smooth.

In general, for a smooth variety $X$, the cohomology ring $H^*(X)$ and the Chow ring $A^*(X)$ are different. Nevertheless, for some nice varieties $X$ these algebras are naturally isomorphic ({\cite[Example 19.1.11]{fulton2013intersection}}).
%The notion of Chow groups for varieties is analogous, see \cite{fulton2013intersection}. Let $A_k(X)$ be the group of algebraic $k$-cycles, formal sums of irreducible subvarieties of $X$ of dimension $k$, modulo rational equivalence. These rational equivalences are generated as divisors of rational functions on $k+1$-dimensional subvarieties of $X$. For the case $X$ is a non-singular variety, we let $A^k(X) = A_{n-k}(X)$, then the product defined using transverse intersection gives $A^*(X)$ the structure of a graded algebra (\cite{fulton2013intersection}[Proposition 8.3]). Again, an irreducible subvariety $V \subset X$ of dimension $k$ generates a class $[V] \subset A^{n-k}(X)$ just as in the case of toric varieties, though this time we do not have the correspondence between closed $T$-invariant subvarieties and cones in the fan. For certain important classes of varieties, there are relations between the Chow ring $A^*(X)$ and the cohomology ring $H^*(X)$. 

\begin{Th}  \label{cellcohomology}
Suppose $X$ is smooth and has a paving by affine cells, then $H^*(X)$ and $A^*(X)$ are naturally isomorphic. 
\end{Th}
The above theorem in particular applies to complete smooth toric varieties and the flag variety $\FL_n$.
%When the variety $X_\Sigma$ is non-singular, we define $A^k(X_\Sigma) = A_{n-k}(X_\Sigma)$, then there is an intersection product on $A^*(X_\Sigma)$ giving it the structure of a commutative graded ring. When $X_\Sigma$ is nonsingular and projective, we have the following description of the Chow ring, see for example \cite{fultonbook}:

When $X=X_\Sigma$ is a smooth complete toric variety, there is a nice description of the Chow ring $A^*(X_\Sigma)$. In this case, for each $k$, the Chow group $A^k(X_\Sigma) = A_{n-k}(X_\Sigma)$ is generated by the orbit closures of codimension $k$. Although not needed in this paper, we state the following well-known result on description of the Chow ring of a smooth complete toric variety (see \cite[Section 5.2]{fultonbook}). 
\begin{Th}   \label{th-coh-smooth-toric-var}
Let $X_\Sigma$ be a smooth complete toric variety. Let $D_1, \ldots, D_r$ be the codimnesion $1$ orbit closures corresponding to rays $\rho_1, \ldots, \rho_r \in \Sigma(1)$.
Then $A^*(X_\Sigma) \cong H^*(X_\Sigma) \cong \Z[D_1,\ldots,D_r]/I$ where $I$ is the ideal generated by the following relations:
\begin{itemize}
    \item[(1)] $D_{i_1} \cdots D_{i_k}$ for all $\rho_{i_1},\ldots,\rho_{i_k}$ not contained in any cone of $\Sigma$ and,
    \item[(2)] $\sum_{i=1}^r \langle u, v_{\rho_i} \rangle D_i$ for all $u \in M$.
\end{itemize}
\end{Th}

%Note that neither of the two varieties we are interested in, $X_{\GZ}$ and $\FL_n$, are nonsingular projective toric varieties.
%This can be combined with the well-known Borel description of the cohomology ring of $\FL_n$, $$H^*(\FL_n) \cong \R [\Lambda_\R]/I_W,$$ where $\Lambda_\R \cong \R^{n-1}$ is the weight lattice tensored with $\R$, and $I_W$ is the ideal generated by non-constant polynomials invariant under the action of the Weyl group $W$. We recall that the map $\Lambda \to H^2(\FL_n)$ given by $\lambda \mapsto c_1(\mathcal{L}_\lambda)$ is additive as $\mathcal{L}_\lambda \otimes \mathcal{L}_\mu = \mathcal{L}_{\lambda + \mu}$, and extends to give an isomorphism of graded rings as both are generated by the above graded pieces. Considering the divisors $D_\lambda$ rather than line bundles $\mathcal{L}_\lambda$ gives us the following isomorphism:
%\begin{equation}\label{picflag}
%    \Lambda \cong H^2(\FL_n) \cong \text{Pic}(\FL_n).
%\end{equation}

There is also a nice description of the ring $A^*(\FL_n) \cong H^*(\FL_n)$ due to Borel. For each weight $\lambda$ let $c_1(\mathcal{L}_\lambda)$ be the divisor class (Chern class) of the line bundle $\mathcal{L}_\lambda$ on $\FL_n$ (see \cite{brion2005lectures}, in particular Remark 1.4.2 in there). %{\color{red} Add reference to Hiller's book Geometry of coxeter groups.} 
\begin{Th}   \label{th-Chow-ring-flag-var}
We have the following:
\begin{itemize}
\item[(1)] The map $\lambda \mapsto c_1(\mathcal{L}_\lambda)$ gives an isomorphism of $A^1(\FL_n) = \Pic(\FL_n)$ with the weight lattice $\Lambda' = \Lambda(\SL(n, \C)) = \Lambda / \Z(1, \ldots, 1)$. 
\item[(2)] $A^*(\FL_n)$ is generated, as an algebra, by $c_1(\mathcal{L}_\lambda)$, $\lambda \in \Lambda$.
\item[(3)] $A^*(\FL_n) \cong \textup{Sym}(\Lambda') / I_W$ where $I_W$ is the ideal generated by non-constant $W$-invariants.
\end{itemize}
\end{Th}
In the proof of our main theorem (Theorem \ref{th-main}) we will need parts (1) and (2) in Theorem \ref{th-Chow-ring-flag-var}.

\begin{Rem}
Alternatively, $H^*(\FL_n, \Q)$ can be viewed as the polytope algebra of the Gelfand-Zetlin family (see \cite[Corollary 5.3]{kaveh2011note}). There it is shown that $$H^*(\FL_n, \Q) \cong \text{Sym}(\Lambda_\Q)/I$$ where $I$ is the ideal of polynomials which, when viewed as differential operators, annihilate the volume polynomial of the Gelfand-Zetlin polytopes. This description of the Chow ring of the flag variety is it is closely related to the proof of Theorem \ref{th-main} but is not directly used there. 
\end{Rem}

%We next recall what is known about $H^*(X_\Sigma)$ for $X_\Sigma$ a toric variety which may not be smooth. Chow groups and rings have been defined for singular toric varieties, the so-called operational Chow ring, by Fulton and MacPherson in \cite{fulton1981categorical}. As in the non-singular case, the Chow group $A_k(X)$ is generated by orbit closures $\overline{V(\sigma)}$ for $\sigma \in \Sigma(n-k)$, however these groups may have torsion. For this reason, we will consider cohomology over $\Q$, or over an arbitrary field $k$. In the case that the fan $\Sigma$ is complete (as it is for $\Sigma_{\GZ}$), we can define $A^k(X) = \text{Hom}(A_k(X),\Z)$ and then the product structure is given by composition of homomorphisms (see \cite{fulton1997intersection}). It has also been shown that, again in the case that $\Sigma$ is complete, over $\Q$ this operational Chow ring is isomorphic to the cohomology ring. The main goal of Fulton and Sturmfels in \cite{fulton1997intersection} is to construct an isomorphism between $A^*(X_\Sigma)$ and another graded ring more combinatorial in nature. This will be discussed in the following section.

We note that the toric variety $X_\GZ$ is not smooth except when $n=1, 2$ and hence we need a more general notion of the Chow ring that applies to non-smooth varieties as well. For a (not necessarily smooth) variety $X$ in \cite{fulton1981categorical} Fulton and MacPherson construct a variant of the Chow ring called the {\it operational Chow ring} or {\it Chow cohomology ring} $A^*(X) = \bigoplus_{k=0}^n A^k(X)$. When $X$ is smooth it coincides with the usual Chow ring. When $X = X_\Sigma$ is a complete toric variety one has $A^k(X_\Sigma) = \textup{Hom}(A_k(X_\Sigma), \Z)$. Moreover, the ring $A^*(X_\Sigma)$ can be described purely in terms of combinatorial data of Minkowski weights, which are certain integer valued functions on the fan $\Sigma$. In Section \ref{sec-n=3} we will use this combinatorial description for some computations in the Chow cohomology of the Gelfand-Zetlin toric variety for $n=3$. Section \ref{sec-MW} reviews the Minkowski weights description of the Chow cohomology ring.  

\section{Some algebra lemmas}
Let $A = \bigoplus_{i=0}^n A^i$ be a graded ring over a field $\k$ which is finite dimensional as a $\k$-vector space and $A^0 \cong A^n \cong \k$. Following \cite{huh2017lefschetz}, we call the graded subalgebra of $A$ generated by $A^1$, the \emph{Lefschetz subalgebra} of $A$. We recall that $A$ has \emph{Poincar\'e duality} if the multiplication maps $$A^i \times A^{n-i} \to A^n \cong \k$$ are non-degenerate for all $i$. Our goal is to compare $A^*(\FL_n) \cong H^*(\FL_n)$, which has Poincar\'e duality, with the algebra $A^*(X_{\GZ})$, which in general does not. We start by observing how to get a Poincar\'e duality algebra from a general graded algebra.

\begin{Lem}\label{pd-lemma}  
Let $A = \bigoplus_{i=0}^n A^i$ with $A^0 \cong A^n \cong \k$. There exists a homogeneous ideal $I \subset A$ such that $A/I$ has Poincar\'e duality and is the smallest homogeneous ideal (with respect to inclusion) with this property. 
\end{Lem}
\begin{proof}
Consider the ideal $I$ generated by all the homogeneous elements $x \in A$ such that $$x \cdot A^{n-\deg(x)} = 0.$$ 
It is straightforward to check that $I$ has the required properties.
%First let us see that the $n$-th graded piece of $I$ is $\{0\}$. 
%Let $z \in I$ with $\deg z = n$, then we must have $z = \sum_i c_i x_i$ where the $x_i$ are generators of $I$ with $\deg x_i = d_i$ and $\deg c_i = n-d_i$. By assumption the $x_i$ satisfy $x_i A^{n-d_i} = 0$, so $c_i x_i = 0$ for all $i$ and hence $z=0$. Suppose for contradiction that $A/I$ does not have Poincar\'e duality, then there is $0 < i < n$ and $0 \neq x \in A^i \setminus I$, such that for all $y \in A^{n-i}$ we have $xy \in I$. As the degrees of $x$ and $y$ are complementary, $\deg xy = n$, so these products lie in the $n$-th graded piece of the ideal $I$. We have shown that the degree $n$ part of $I$ is trivial, and hence $xy \in I$ with $\deg xy = n$ implies that $xy = 0$. This implies that $x \in I$ which is a contradiction. Thus $A/I$ does have Poincar\'e duality. Minimality of $I$ with respect to inclusion is obvious. 
%We next show that $I$ is the minimal such homogeneous ideal. Suppose not, then there exists homogeneous ideal $J$ such that $A/J$ has Poincar\'e duality and also nonzero $x \in (I \setminus J).$ As both ideals are homogeneous, we can take $x$ to be homogeneous say of degree $i$. Since $x \in I$ we must have $x \cdot A^{n-i} = 0$. Since $x \notin J$, it corresponds to a non-trivial coset $\bar x$ in $A/J$. However, this $\bar x$ satisfies $\bar x \cdot (A/J)^{n-i} = 0$, contradicting Poincar\'e duality for $A/J$.
\end{proof}
We call the algebra $A/I$ in Lemma \ref{pd-lemma}, the \emph{Poincar\'e duality quotient} $\text{PD}(A)$ of $A$.
We next recall a useful algebra fact 
(see \cite[Theorem 1.1]{kaveh2011note} and \cite[Exercise 21.7]{eisenbud}) which we will need later. It states that a Poincar\'e duality algebra is determined by its top power polynomial. 
%For the sake of completeness we include a proof here. 
\begin{Th}\label{algebraKaveh}
Let $A = \bigoplus_{i=0}^n A^i$ be a finite dimensional graded algebra over a field $\k$ which is generated by $A^1$, satisfies $A^0 \cong \k \cong A^n$, and has Poincar\'e duality.
Fix a basis $\{a_1, \ldots, a_r\}$ for $A^1$, and consider the polynomial $P: \k^r \to \k$ defined by $$P(x_1,\ldots,x_r) = (x_1 a_1 + \cdots + x_r a_r)^n \in A^n \cong \k.$$ 
Then we have an isomorphism of graded algebras $$A \cong \k[\partial_1, \ldots, \partial_r]/I$$ where $\partial_i = \frac{\partial}{\partial x_i}$, and $I$ is the ideal of polynomials in the operators $\partial_1, \ldots, \partial_r$ which annihilate $P$. The isomorphism sends each $a_i$ to the image of $\partial_i$ in $\k[\partial_1, \ldots, \partial_r] / I$. 
\end{Th}

A generalization of Theorem 8.1 for commutative algebras $A$ with Poincar\'e duality that are not necessarily generated by $A^0 \cong \k$ and $A^1$ can be found in \cite{EKKh}.

We now use Theorem \ref{algebraKaveh} to to prove the following key lemma required in the proof of our main result (Theorem \ref{th-main}).
\begin{Lem}\label{algLemma}
Suppose $A = \bigoplus_{i=0}^n A^i$ and $B = \bigoplus_{i=0}^n B^i$ are $\k$-algebras which are finite dimensional $\k$-vector spaces and have the following properties:
\begin{itemize}
\item[(1)] $A^0 \cong A^n \cong B^0 \cong B^n \cong \k$. 
\item[(2)] $A$ and $B$ are generated in degree one. 
\item [(3)] $A$ has Poincar\'e duality. 
\item[(4)] There exists a linear isomorphism $\varphi: A^1 \to B^1$ such that for all $a_1, \ldots, a_n \in A^1$ we have: $$a_1 \cdots a_n = \varphi(a_1) \cdots \varphi(a_n)$$ using fixed isomorphisms $A^n \cong \k \cong B^n$. 
\end{itemize}
Then $\varphi$ extends to give a $\k$-algebra isomorphism $\tilde \varphi$ between $A$ and the Poincar\'e duality quotient of $B$.%, i.e., $$\tilde \varphi: A \overset{\cong}{\to} \text{PD}(B).$$ %really this will end up being the Lefschetz algebra of X_{\GZ}, which is defined to be the subalgebra of the chow ring generated by divisors (i.e., in degree 1)
%statement changed because taking the Lefschetz subalgebra might kill the assumption of the degree n piece being isomorphic to k. It was easiest to just apply this lemma to the lefschetz subalgebra of the chow ring of X_{\GZ} rather than taking the subalgebra within this lemma.
\end{Lem}
\begin{proof}
We apply Theorem \ref{algebraKaveh} to $A$ and to the Poincar\'e duality quotient $\text{PD}(B)$. Since $A$ already satisfies the conditions of Theorem \ref{algebraKaveh} we know that $A \cong \k[\partial_1, \ldots, \partial_r]/I$ where $r = \dim_\k(A^1)$ and $I$ is the annihiliator of the top power polynomial $P$ described in Theorem \ref{algebraKaveh}. 
We need to show that $\text{PD}(B)$ also satisfies these conditions. First note that $B^0 \cong \k \cong B^n$ so the multiplication $B^0 \times B^n \to B^n \cong \k$ is already non-degenerate and thus the ideal $I$ in Lemma \ref{pd-lemma} contains neither $B^0$ nor $B^n$. This gives us $\text{PD}(B)^0 \cong \k \cong \text{PD}(B)^n.$ Also, by construction $\text{PD}(B)$ has Poincar\'e duality. Finally, $\text{PD}(B)$ is generated in degree one since $B$ is generated in degree $1$. Now consider the map on degree one pieces: $$A^1 \overset{\varphi}{\to} B^1 \overset{q}{\to} \text{PD}(B)^1,$$ where $q$ is the quotient map. It suffices to show $\tilde \varphi := q \circ \varphi: A^1 \to \text{PD}(B)^1$ is an isomorphism. Since $\varphi$ is an isomorphism and $q$ is surjective, $\tilde \varphi$ is surjective and we only need to verify injectivity. Suppose for contradiction that some nonzero $a \in A^1$ has image $\tilde \varphi(a) = q ( \varphi(a)) = 0$. Then $b = \varphi(a)$ is in the ideal in Lemma \ref{pd-lemma}, so it is a linear combination of the $x_i$ satisfying $x_i \cdot B^{n-\deg(x_i)} = 0.$ Since $b \in B^1$, the $x_i$ must be in degree $0$ or $1$. One knows that $B^0 \cap I = \{0\}$, so we can only have $x_i \in B^1$. It follows that $b \cdot B^{n-1} = 0$. But the assumption (4) then implies that $a \cdot A^{n-1} = 0$ which contradicts that $A$ has Poincar\'e duality. Thus $\text{PD}(B)$ satisfies the conditions required for Theorem \ref{algebraKaveh}, and hence $\text{PD}(B) \cong \k[\partial_1,\ldots,\partial_r]/I$. We have already seen that $A$ is isomorphic to this quotient algebra and thus $A \cong \text{PD}(B)$.
\end{proof}
%we will eventually apply this, but with n = N = n(n-1)/2, dim of GZ fan etc

% subalgebra first, then give ring PD via Poincar\'e duality quotient

\section{Main theorem}
We now state and prove our main theorem relating the cohomology ring of the flag variety $\FL_n$ and the Chow cohomology ring of the toric variety $X_{\GZ}$ associated to the GZ fan $\Sigma=\Sigma_{\GZ}$.
\begin{Th}   \label{th-main}
The cohomology ring $H^*(\FL_n, \Q) \cong A^*(\FL_n, \Q)$ is isomorphic to the Poincar\'e duality quotient of the Lefschetz subalgebra of $A^*(X_{\GZ}, \Q)$. For each dominant weight $\lambda$, the isomorphism sends the divisor class of the line bundle $\mathcal{L}_\lambda$ on $\FL_n$ to the image of the cohomology class in $X_\GZ$ associated to the GZ polytope $\Delta_\lambda$.
\end{Th}
\begin{proof}
We claim that there is an isomorphism of groups $A^1(\FL_n) \cong A^1(X_{\GZ})$. One knows that $A^1(\FL_n) = A_{N-1}(\FL_n) = \Pic(\FL_n) \cong \Lambda(\SL(n, \C)) = \Lambda(\GL(n, \C)) / \Z(1, \ldots, 1)$.
Also for a complete toric variety $X_\Sigma$, where $\Sigma$ is a complete fan in $\R^N$, the Chow cohomology group $A^1(X_\Sigma)$ is naturally isomorphic to $\Pic(X_\Sigma)$ (see \cite[Corollary 3.4]{fulton1997intersection}). Now the claim follows from Corollary \ref{cor-additive}.

%Moreover, it is well-known that $\Pic(X_\Sigma)$ is isomorphic to the group $\textup{PL}(\Sigma, \Z^N)$ of integer piecewise linear functions on $\Sigma$ modulo integer linear functions. The group $\textup{PL}(\Sigma, \Z^N)$ consists of functions on $\R^N$ that are linear on each cone in $\Sigma$ and have integer values on $\Z^N$. This group contains a copy of $\Z^N$ as the subgroup of all linear functions from $\Z^N$ to $\Z$. 

%By Corollary \ref{cor-additive}, $\Pic(X_\Sigma)$ can be identified with the quotient group $\mathcal{P}(\Sigma) / \Z^N$. Recall that $\mathcal{P}(\Sigma)$ is the subgroup of virtual lattice polytopes generated by lattice polytopes whose normal fan is $\Sigma$. It naturally contains $\Z^N$ as a subgroup (see paragraph before Corollary \ref{cor-additive}). If we apply the above to the Gelfand-Zetlin fan $\Sigma_\GZ$, from Corollary \ref{cor-additive} it follows that $\Pic(\FL_n) \cong \Lambda' \cong \mathcal{P} / \Z^N \cong \Pic(X_\GZ)$ as required. 

One knows that for an $N$-dimensional toric variety $X_\Sigma$, under the isomorphism $A^1(X_\Sigma) \cong \Pic(X_\Sigma)$ the top product of an element in $A^1(X_\Sigma) \cong \Pic(X_\Sigma)$ coincides with the self-intersection number of the corresponding divisor in $\Pic(X_\Sigma)$. Applying this to the Gelfand-Zetlin toric variety $X_\GZ$, from Propositions \ref{degreeToric} and \ref{flagDegree}, we now conclude that the isomorphism $\Pic(\FL_n) = \Pic(X_\GZ)$ respects the multiplication, i.e., it satisfies the assumption (4) in Lemma \ref{algLemma} (alternatively this can be deduced from \cite[Theorem 4.3 and Corollary 4.5]{Jensen-Yu}). Applying Lemma \ref{algLemma} to $A = A^*(\FL_n)$ and $B =$ the Lefschetz subalgebra of $A^*(X_\GZ)$ finishes the proof. 

\end{proof}

\section{Minkowski weights}   \label{sec-MW}
In this section we recall the description of the Chow cohomology ring of a toric variety in terms of Minkowski weights (see \cite{fulton1997intersection}, see also \cite{Kazarnovskii}). We will use it in Section \ref{sec-n=3} to compute the Gelfand-Zetlin Chow cohomology ring for $n=3$. 
Let $\Sigma$ be a complete fan in $N$. Recall that $\Sigma(k)$ is the set of cones of dimension $k$ in $\Sigma$.
\begin{Def}
A function $c: \Sigma(n-k) \to \Z$ is a \emph{Minkowski weight} of codimension $k$ on $\Sigma$ if it satisfies the \emph{balancing condition} for all $\tau \in \Sigma(n-k-1)$: 
\begin{equation}\label{balancing}
\sum_{\sigma \in \Sigma(n-k), \sigma \supset \tau} \langle u, n_{\sigma,\tau} \rangle c(\sigma) = 0, \quad \forall u \in M(\tau) := M \cap \tau^\perp.
\end{equation}
Here $n_{\sigma,\tau}$ is a lattice point in $\sigma$ which generates the rank $1$ lattice $N_\sigma/N_\tau$, the quotient of the lattices spanned by $\sigma \cap N$ and $\tau \cap N$ respectively. %The above equation must be satisfied for all $u \in M(\tau),$ the lattice perpendicular to the span of $\tau$.
\end{Def}

Let $MW^k$ denote the set of all Minkowski weights of codimension $k$. For two Minkowski weights $c \in MW^p$ and $\tilde c \in MW^q$, the product $c \cup \tilde c \in MW^{p+q}$ is defined by: $$(c \cup \tilde c)(\gamma) = \sum_{(\sigma,\tau) \in \Sigma(n-p) \times \Sigma(n-q)} m_{\sigma,\tau}^\gamma c(\sigma) \tilde c(\tau), \quad \forall \gamma \in \Sigma(n-p-q),$$ where $m_{\sigma,\tau}^\gamma = [N: N_\sigma + N_\tau]$, and the sum is over all pairs of cones $(\sigma,\tau)$ which both contain $\gamma$ and $\sigma$ meets $\tau + v$ for fixed generic vector $v$ (see \cite[Theorem 4.2]{fulton1997intersection}).

In \cite{fulton1997intersection} an isomorphism between the ring of Minkowski weights and the operational Chow ring of a complete toric variety $X_\Sigma$ is given. In fact it is shown that $MW^k \cong A^k(X_\Sigma)$ (see \cite[Theorem 3.1]{fulton1997intersection}). In particular:
\begin{equation}\label{pictoric}
\text{Pic}(X_\Sigma) \cong A^1(X_\Sigma).
\end{equation}

%In \cite[Corollary 4.6]{katz2008} it is proven that for any toric variety there is an isormorphism $\text{Pic}(X) \cong A^1(X)$. The multiplication of Minkowski weights described above in fact gives $A^*(X_\Sigma)$ the structure of a graded ring isomorphic to the operational Chow ring.

\begin{Ex}[Hypersimplex]
The following is an example of a fan where the ring $MW^*$ is not generated by $MW^1$ (see \cite[Example 3.5]{fulton1997intersection} or \cite[Example 4.2]{katz2008}). Consider the fan $\Sigma_H$ over the cube in $\R^3$ with vertices $(\pm 1, \pm 1, \pm 1)$. The rays in the fan $\Sigma_H$ are:
\begin{align*}
    \rho_1 = \langle 1, 1, 1 \rangle \quad & \quad \rho_5 = -\rho_1\\
    \rho_2 = \langle 1, 1, -1 \rangle \quad & \quad \rho_6 = -\rho_2\\
    \rho_3 = \langle 1, -1, 1 \rangle \quad & \quad \rho_7 = -\rho_3\\
    \rho_4 = \langle -1, 1, 1 \rangle \quad & \quad \rho_8 = -\rho_4
\end{align*}
%The $2$-dimensional cone spanned by $\rho_i$ and $\rho_j$ will be denoted $\sigma_{ij}$, and similarly the $3$-dimensional cone spanned by $\rho_i$, $\rho_j$ and $\rho_k$ will be denoted $\sigma_{ijk}$.
One computes that $MW^1 \cong \Z$ and $MW^2 \cong \Z^5$. Thus $MW^*$ is not generated by $MW^1$. 

\end{Ex}

\section{Gelfand-Zetlin example, $n=3$}   \label{sec-n=3}
In this section we compute the Chow cohomology ring of $X_{\GZ}$ for $n=3$ using the Minkowski weights and show that while it is generated in degree $1$, it does not have Poincar\'e duality. We consider the GZ polytope of the weight $\lambda = (-1,0,1)$ for ease of computation. The polytope $\Delta_\lambda$ is defined by the following array of inequalities 
$$\begin{array}{ccccc}
    -1 && 0 && 1 \\
    &x && y & \\
    && z &&
\end{array}$$
and has normal fan $\Sigma_{\GZ}$ as in Figure \ref{fig:gzfan}. We enumerate the rays as follows:
$$\begin{array}{ccc}
    \rho_1 = (1,0,0) & \rho_3 = (0,1,0) & \rho_5 = (1,0,-1)\\
    \rho_2 = (-1,0,0) & \rho_4 = (0,-1,0) & \rho_6 = (0,-1,1).
\end{array}$$

\begin{figure}[h]
    \centering
    \includegraphics[scale=.4]{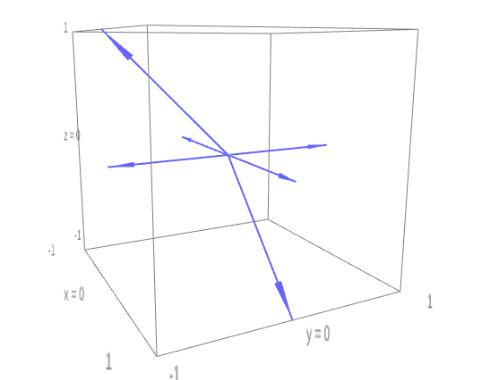}
    \caption{Rays of $\Sigma_{\GZ}$ for $n=3$.}
    \label{fig:gzfan}
\end{figure}

Likewise, we let $\sigma_{ij}$ denote the 2-dimensional cone spanned by rays $\rho_i$ and $\rho_j$: %Note that sigma_{14} would correspond to the line segment on a cube which has been contracted to a single point, the "bad" point in the GZ polytope.
$$\begin{array}{cccc}
\sigma_{13} & \sigma_{23} & \sigma_{24} &\\
\sigma_{15} & \sigma_{25} & \sigma_{35} & \sigma_{45}\\
\sigma_{16} & \sigma_{26} & \sigma_{36} & \sigma_{46}
\end{array}$$
Similarly, the collection of 3-dimensional cones are:
$$\begin{array}{cccc}
\gamma_{135} & \gamma_{235} & \gamma_{245} & \gamma_{1456}\\
\gamma_{136} & \gamma_{236} & \gamma_{246} &
\end{array}$$

We now compute $MW^k$, $k=0, \ldots, 3$. A Minkowski weight in $MW^3$ is any map $\{0\} \to \Z$ and hence $MW^3 \cong \Z$. A Minkowski weight $c \in MW^2$ is a function on rays $\rho_i$. Let $c(\rho_i) = c_i$, then the single relation coming from the cone $\tau = 0$ is given by $\sum_{i=1}^6 c_i v_{\rho_i} = 0$.
From this we get the three relations:
\begin{align*}
    c_1 - c_2 + c_5 &= 0\\
    c_3 - c_4 - c_6 &= 0\\
    -c_5 + c_6 &= 0.
\end{align*}
We see from this that any weight $c \in MW^2$ is determined by its values on three rays $c(\rho_2) = a$, $c(\rho_4) = b$ and $c(\rho_6)=c$. Thus $MW^2 \cong \Z^3$.
%, then
%\begin{equation}\label{mw2}
%\begin{array}{ccc}
%c(\rho_1) = a-c & c(\rho_3) = b+c & c(\rho_5) = c\\ 
%c(\rho_2) = a   & c(\rho_4) = b   & c(\rho_6) = c.
%\end{array}
%\end{equation}

%Next, we examine $MW^1$. These are functions on codimension $1$ cones $\sigma_{ij}$. 
Next take $c \in MW^1$. It is a function on codimension $1$ cones $\sigma_{ij}$. Let $c(\sigma_{ij}) = c_{ij}$. %Then a weight of codimension $1$ is given by the data
%$$\begin{array}{cccc}
%c_{13} & c_{23} & c_{24} &\\
%c_{15} & c_{25} & c_{35} & c_{45}\\
%c_{16} & c_{26} & c_{36} & c_{46}
%\end{array}$$
%subject to relations coming from the rays $\{\rho_i\}$.
The relations among the $c_{ij}$ come from the rays. The relation for $\tau = \rho_1$ involves the cones $\sigma_{13}, \sigma_{15}$ and $\sigma_{16}$. Let $n_{\sigma \tau}$ be the lattice point in $\sigma$ which generates the one-dimensional lattice $N_\sigma/N_\tau$. %The relation will be a vector equation in the vector space perpendicular to $\rho_1 = (1,0,0)$. 
We compute: $$n_{13} = (0,1,0), \quad n_{15} = (0,0,-1), \quad \mbox{ and } \quad n_{16} = (0, -1, 1)$$ where all vectors are considered modulo $\rho_1 = (1, 0, 0)$. The balancing condition then becomes $$c_{13} (0,1,0) + c_{15} (0,0,-1) + c_{16} (0, -1, 1) = (0, 0, 0)$$ which implies $c_{13} = c_{15} = c_{16}$.
Similar computations for the other rays yield the following results:
\begin{align*}
    c_{13} &= c_{15} = c_{16} = c_{25} = c_{26}\\
    c_{24} &= c_{35} = c_{36} = c_{45} = c_{46}\\
    c_{23} &= c_{13} + c_{24}
\end{align*}
For later computations, we let:
\begin{align*}
    a &= c_{13} = c_{15} = c_{16} = c_{25} = c_{26}\\
    b &= c_{24} = c_{35} = c_{36} = c_{45} = c_{46}\\
    c_{23} &= a+b.
\end{align*}

Finally, a weight $c \in MW^0$ is a function on top-dimensional cones subject to relations coming from $2$-dimensional cones. Each $2$-dimensional cone $\sigma_{ij}$ separates two top-dimensional cones, and the corresponding relation gives equality between the values of $c$ on each pair of top-dimensional cones. Hence $MW^0 \cong \Z$ as the value of $c$ on each $3$-dimensional cone must be the same. In summary, we have the following:
\begin{align*}
    MW^0 & \cong \Z\\
    MW^1 & \cong \Z^2\\
    MW^2 & \cong \Z^3\\
    MW^3 & \cong \Z.
\end{align*}
Before understanding the product structure on $MW^*$, it is already clear that this ring does not have Poincar\'e duality as the rank of $MW^2$ is greater than that of $MW^1$.

%Next we examine the product structure on $MW^*(X_{\GZ})$. 
Recall from Section \ref{sec-MW} that for weights $c \in MW^p$, $\tilde c \in MW^q$, their product is a function on cones of codimension $p+q$, and its value on a cone $\gamma \in \Sigma(3-p-q)$ is given by
\begin{equation}\label{product}
(c \cup \tilde c)(\gamma) = \sum_{(\sigma, \tau)} m_{\sigma \tau}^
\gamma c(\sigma) \tilde c(\tau),
\end{equation}
where the sum is over certain pairs $(\sigma, \tau) \in \Sigma(3-p) \times \Sigma(3-q)$ and $m_{\sigma \tau}^\gamma = [N: N_\sigma + N_\tau]$.
%provided that:
%\begin{enumerate}
%\item $\gamma \subset \sigma \cap \tau$,
%\item $\sigma$ meets $\tau + v$ for a generic fixed $v \in N$,
%\end{enumerate}
%otherwise $m_{\sigma \tau}^\gamma = 0.$ %Recall also that $\Sigma(n-p)$ is the set of cones in $\Sigma$ of dimension $n-p$.
We compute products of Minkowski weights in our example to determine whether $MW^*(X_{\GZ})$ is generated in degree $1$. Let $c, \tilde c \in MW^1(X_{\GZ})$ with:
\begin{align*}
    c: \{\sigma_{13}, \sigma_{15}, \sigma_{16}, \sigma_{25}, \sigma_{26}\} &\mapsto a\\
    c: \{\sigma_{24}, \sigma_{35}, \sigma_{36}, \sigma_{45}, \sigma_{46}\} &\mapsto b\\ 
    c: \{\sigma_{23}\} &\mapsto a+b\\
    \tilde c: \{\sigma_{13}, \sigma_{15}, \sigma_{16}, \sigma_{25}, \sigma_{26}\} &\mapsto \tilde a\\
    \tilde c: \{\sigma_{24}, \sigma_{35}, \sigma_{36}, \sigma_{45}, \sigma_{46}\} &\mapsto \tilde b\\ 
    \tilde c: \{\sigma_{23}\} &\mapsto \tilde a+\tilde b.
\end{align*}
The Minkowski weight $c \cup \tilde c \in MW^2$ is evaluated on rays and from the arguments above it is enough to determine the value of this weight on the rays $\rho_2$, $\rho_4$ and $\rho_5$. %We begin by examining $(c \cup \tilde c)(\rho_2)$ via Equation \eqref{product}.  
Moreover, in Equation \eqref{product} for $(c \cup \tilde c)(\rho_2)$ the sum is over all pairs $(\sigma,\tau) \in \Sigma(2)\times \Sigma(2)$ where $\sigma$ and $\tau$ both contain $\rho_2$ and $\sigma$ meets $\tau + v$ for a generic fixed $v \in N$. The cones in $\Sigma(2)$ which contain $\rho_2$ are $\{\sigma_{23},\sigma_{24},\sigma_{25},\sigma_{26}\}$, so $\sigma, \tau$ will come from this collection. Since all these cones involve $\rho_2 = (-1,0,0)$, we can sketch the relevant cones in the $yz$-plane where for example $\sigma_{23}$ can be viewed as $\rho_3 = (1,0)$. 
\begin{figure}[h]
    \centering
    \includegraphics[scale=.4]{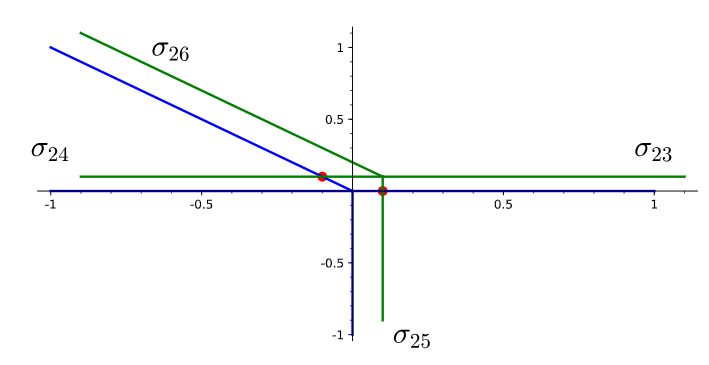}
    \caption{Intersection of
    $\sigma$ and $\tau + v$}
    \label{fig:intersection}
\end{figure}
In Figure \ref{fig:intersection}, we see the cones for $c$ in blue, and for $\tilde c$ in green using a shift of $v = (0.1,0.1,0.1).$ Then there are two pairs $(\sigma,\tau)$ which meet for this vector $v$, either $(\sigma,\tau) = (\sigma_{23},\sigma_{25})$ or $(\sigma,\tau) = (\sigma_{26},\sigma_{24})$. The last ingredient required to compute this product are the coefficients $m_{\sigma \tau}^{\rho_2}$ for the sum. %Recall $m_{\sigma \tau}^\gamma$ is $[N:N_{\sigma} + N_\tau]$. 
In both cases, $N_\sigma+N_\tau = N$ so $m_{\sigma \tau}^{\rho_2}=1$. Thus we have
\begin{align*}
    (c \cup \tilde c)(\rho_2) &= c(\sigma_{23})\tilde c(\sigma_{25}) + c(\sigma_{26}) \tilde c (\sigma_{24})\\
    &= (a+b) \tilde a + a(\tilde b)\\
    &= a \tilde a + b \tilde a + a \tilde b.
\end{align*}

Similar computations for $(c \cup \tilde c)(\rho_4)$ and $(c \cup \tilde c)(\rho_5)$ yield:
\begin{align*}
    (c \cup \tilde c)(\rho_4) &= b \tilde b\\
    (c \cup \tilde c)(\rho_5) &= b \tilde a + a \tilde b.
\end{align*}
Thus we see that the products $c \cup \tilde c$ in fact generate the entire $3$-dimensional space $MW^2,$ and hence $MW^*$ for $\Sigma_{\GZ}$ is generated in degree $1$ for the case $n=3$.

Finally, the second author has written a Sage code which shows that for $n=4$, $5$, the ring $MW^*$ of $\Sigma_\GZ$ is not generated in degree $1$, and moreover its Lefschetz subalgebra does not have Poincar\'e duality. It can be found at {\sf https://github.com/evillella/minkowski}. Also see the appendix in \cite{EliseThesis}.

\bibliography{thisbib}

\begin{thebibliography}{BCFKvS00}

\bibitem[BCFKvS00]{Batyrev}
Victor~V. Batyrev, Ionu\c{t} Ciocan-Fontanine, Bumsig Kim, and Duco van
  Straten.
\newblock Mirror symmetry and toric degenerations of partial flag manifolds.
\newblock {\em Acta Math.}, 184(1):1--39, 2000.

\bibitem[Bri05]{brion2005lectures}
Michel Brion.
\newblock Lectures on the geometry of flag varieties.
\newblock In {\em Topics in cohomological studies of algebraic varieties},
  pages 33--85. Springer, 2005.

\bibitem[CLS11]{Cox-Little-Schenck}
David~A. Cox, John~B. Little, and Henry~K. Schenck.
\newblock {\em Toric varieties}, volume 124 of {\em Graduate Studies in
  Mathematics}.
\newblock American Mathematical Society, Providence, RI, 2011.

\bibitem[Eis95]{eisenbud}
David Eisenbud.
\newblock {\em Commutative algebra with a view toward algebraic geometry},
  volume 150.
\newblock Springer-Verlag, New York, 1995.

\bibitem[EKK]{EKKh}
A.~Esterov, B.~Kazarnovskii, and A.~G. Khovanskii.
\newblock Newton polyhedra and tropical geometry.
\newblock {\em Russian Math. Surveys, to appear}.

\bibitem[FM81]{fulton1981categorical}
William Fulton and Robert MacPherson.
\newblock {\em Categorical framework for the study of singular spaces}, volume
  243.
\newblock American Mathematical Soc., 1981.

\bibitem[FS97]{fulton1997intersection}
William Fulton and Bernd Sturmfels.
\newblock Intersection theory on toric varieties.
\newblock {\em Topology}, 36(2):335--353, 1997.

\bibitem[Ful93]{fultonbook}
William Fulton.
\newblock {\em Introduction to toric varieties}, volume no. 131.
\newblock Princeton University Press, Princeton, N.J, 1993.

\bibitem[Ful13]{fulton2013intersection}
William Fulton.
\newblock {\em Intersection theory}, volume~2.
\newblock Springer Science \& Business Media, 2013.

\bibitem[GZ50]{Gelfand:1950ihs}
Israel~M. Gelfand and Michael~L. Zetlin.
\newblock {Finite-dimensional representations of the group of unimodular
  matrices}.
\newblock {\em Dokl. Akad. Nauk Ser. Fiz.}, 71:825--828, 1950.

\bibitem[HW17]{huh2017lefschetz}
June Huh and Botong Wang.
\newblock Lefschetz classes on projective varieties.
\newblock {\em Proceedings of the American Mathematical Society},
  145(11):4629--4637, 2017.

\bibitem[JY16]{Jensen-Yu}
Anders Jensen and Josephine Yu.
\newblock Stable intersections of tropical varieties.
\newblock {\em J. Algebraic Combin.}, 43(1):101--128, 2016.

\bibitem[Kav11]{kaveh2011note}
Kiumars Kaveh.
\newblock Note on cohomology rings of spherical varieties and volume
  polynomial.
\newblock {\em Journal of Lie Theory}, 21(2):263--283, 2011.

\bibitem[Kaz03]{Kazarnovskii}
B.~Ya. Kazarnovski\u{\i}.
\newblock c-fans and {N}ewton polyhedra of algebraic varieties.
\newblock {\em Izv. Ross. Akad. Nauk Ser. Mat.}, 67(3):23--44, 2003.

\bibitem[KM05]{Kogan-Miller}
Mikhail Kogan and Ezra Miller.
\newblock Toric degeneration of {S}chubert varieties and {G}elfand-{T}setlin
  polytopes.
\newblock {\em Adv. Math.}, 193(1):1--17, 2005.

\bibitem[KP08]{katz2008}
Eric Katz and Sam Payne.
\newblock Piecewise polynomials, {M}inkowski weights, and localization on toric
  varieties.
\newblock {\em Algebra Number Theory}, 2(2):135--155, 2008.

\bibitem[KST12]{KST}
V.~A. Kirichenko, E.~Yu. Smirnov, and V.~A. Timorin.
\newblock Schubert calculus and {G}elfand-{T}setlin polytopes.
\newblock {\em Uspekhi Mat. Nauk}, 67(4(406)):89--128, 2012.

\bibitem[KV18]{kaveh2018notion}
Kiumars Kaveh and Elise Villella.
\newblock On a notion of anticanonical class for families of convex polytopes.
\newblock {\em arXiv preprint arXiv:1802.06674}, 2018.

\bibitem[Oko98]{Okounkov-Newton-polytopes}
Andre\u{\i} Okounkov.
\newblock Multiplicities and {N}ewton polytopes.
\newblock In {\em Kirillov's seminar on representation theory}, volume 181 of
  {\em Amer. Math. Soc. Transl. Ser. 2}, pages 231--244. Amer. Math. Soc.,
  Providence, RI, 1998.

\bibitem[Pay06]{Payne}
Sam Payne.
\newblock Equivariant {C}how cohomology of toric varieties.
\newblock {\em Math. Res. Lett.}, 13(1):29--41, 2006.

\bibitem[Vil19]{EliseThesis}
Elise Villella.
\newblock {\em Gelfand-Zetlin polytopes and the geometry of flag varieties}.
\newblock PhD thesis, University of Pittsburgh, 2019.

\end{thebibliography}
\bibliographystyle{alpha}

%\begin{thebibliogrpahy}{99}
%\bibitem[Eliyashev]{Eliyashev}
%\end{thebibliography}
\end{document}